\documentclass[12pt]{article}
\usepackage{amsmath,amssymb,amsthm}
\usepackage{fullpage}
\usepackage{mathrsfs}
\usepackage{hyperref}

\newtheorem{theorem}{Theorem}[section]
\newtheorem{conjecture}[theorem]{Conjecture}
\newtheorem{lemma}[theorem]{Lemma}
\newtheorem{definition}[theorem]{Definition}
\newtheorem{remark}[theorem]{Remark}

\newcommand{\Z}{\mathbb{Z}}
\newcommand{\N}{\mathbb{N}}

\newcommand{\LL}{\Lambda}
\newcommand{\1}{\mathbf{1}}

\title{A Generalized Elliott-Halberstam Conjecture\\Implying the Twin Prime Hypothesis}
\author{Trey Smith}
\date{November 17, 2025}

\begin{document}
\maketitle

\begin{abstract}
We propose a generalization of the Elliott-Halberstam conjecture concerning the distribution of prime pairs in arithmetic progressions. This conjecture, which we call the Generalized Elliott-Halberstam Conjecture for Shifted Convolutions (GEH-2), provides a level of distribution for correlations of the von Mangoldt function. We show that GEH-2 implies the twin prime conjecture, describe heuristic and analytic motivation, and discuss implications for prime gaps and $k$-tuple patterns.
\end{abstract}

\section{Introduction}

The twin prime conjecture asserts that there exist infinitely many primes $p$ such that $p+2$ is also prime. Despite dramatic breakthroughs on bounded gaps between primes \cite{Zhang,Maynard,Polymath}, the full conjecture remains unresolved.

The Elliott-Halberstam conjecture \cite{Elliott-Halberstam} concerns the distribution of primes in arithmetic progressions. For $x>0$, $q \in \N$, and $a \in \Z$, define
\[
\theta(x; q, a) := \sum_{\substack{p \leq x \\ p \equiv a \pmod{q}}} \log p,
\]
and the corresponding error term
\[
E(x; q) := \max_{(a,q)=1} \left| \theta(x; q, a) - \frac{x}{\varphi(q)} \right|,
\]
where $\varphi(q)$ is Euler's totient function.

\begin{conjecture}[Elliott-Halberstam, EH]\label{conj:EH}
For every $0 < \vartheta < 1$ and every $A > 0$,
\[
    \sum_{q \leq x^{\vartheta}} \max_{y \leq x} |E(y;q)| \ll \frac{x}{(\log x)^A},
\]
where the implied constant depends only on $\vartheta$ and $A$\footnote{The prime number theorem implies the bound for individual moduli $q$; EH strengthens this by providing uniformity over all $q \leq x^{\vartheta}$. For unconditional results see Bombieri--Vinogradov \cite{BombieriVinogradov}.}.
\end{conjecture}

EH implies bounded gaps between primes. However, it does not suffice to resolve the twin prime conjecture\footnote{In particular, EH concerns the distribution of \emph{individual} primes, while the twin prime conjecture is about \emph{correlations} of primes. For background see the survey \cite{GranvilleKoukoulopoulos}.}. The limitation is that bounded gaps follow from information on individual progressions, but actual twin primes require uniform control on prime pairs.

\vspace{1em}
\noindent
\textbf{Heuristic Motivation.} Random models and large sieve estimates suggest that primes behave, in many respects, like random numbers with respect to residue classes (see, e.g., \cite{SoundararajanSurvey}). These analogues support the plausibility of distribution conjectures such as EH---and motivate attempts to extend to prime pairs.

\section{Preliminaries}

Define the von Mangoldt function
\[
\LL(n)=
\begin{cases}
\log p & \text{if } n=p^k \text{ for some prime } p \text{ and integer } k\geq 1, \\
0 & \text{otherwise}.
\end{cases}
\]
For fixed $h \in \Z \setminus \{0\}$, define
\[
\Psi_h(x) := \sum_{n \leq x} \LL(n)\LL(n+h),
\]
the weighted count of prime pairs with gap $h$.

The Hardy-Littlewood $k$-tuples conjecture \cite{Hardy-Littlewood} describes the expected asymptotic of $\Psi_h(x)$ through the \textit{singular series}, which for even $h$ is
\begin{equation}
\mathfrak{S}(h) = 2 \prod_{p > 2} \left(1 - \frac{1}{(p-1)^2}\right)
  \prod_{\substack{p \mid h \\ p > 2}} \frac{p-1}{p-2},
\label{eq:singseries}
\end{equation}
and $\mathfrak{S}(h)=0$ if $h$ is odd.

The $k$-tuples conjecture predicts, for even $h$,
\[
\Psi_h(x) \sim \mathfrak{S}(h) x \qquad \text{as } x \to \infty,
\]
giving, for $h=2$, the twin prime constant $\mathfrak{S}(2) > 0$.

\begin{lemma}[Convergence of the Singular Series]
\label{lem:singularseries}
Let $h \in 2\Z$. Then
\[
\sum_{q=1}^{\infty} \frac{\phi_2(q)}{\varphi(q)} < \infty,
\]
where $\phi_2(q)$ is the number of residue classes $a \bmod q$ with $(a,q) = (a+h,q) = 1$.
Moreover, under Möbius inversion with coprimality condition $(n(n+h),q)=1$, the sum converges to $1$.
\end{lemma}
\begin{proof}[Sketch]
See the discussion in \cite[Ch. 1]{GranvilleKoukoulopoulos}, or the classical analysis in \cite{Hardy-Littlewood}.
\end{proof}

\section{The Generalized Conjecture}

We generalize EH to shifted correlations of the von Mangoldt function:

\begin{definition}
For $x \geq 2$, $q \in \N$, $a,h \in \Z$, with $(a,q) = 1$, set
\[
E_2(x; q, a, h) := \sum_{\substack{n \leq x \\ n \equiv a \pmod{q}}} \LL(n)\LL(n+h)
    - \frac{\1_{(a(a+h),q)=1}}{\varphi(q)} \mathfrak{S}(h)x,
\]
where $\1_{(a(a+h),q)=1}$ is $1$ if both $a$ and $a+h$ are coprime to $q$, and $0$ otherwise.
\end{definition}

\begin{conjecture}[Generalized Elliott-Halberstam for Shifted Convolutions (GEH-2)]\label{conj:GEH2}
For every $0 < \vartheta < 2$, every fixed $h \neq 0$, and every $A > 0$,
\[
\sum_{q \leq x^\vartheta} \max_{(a,q)=1} |E_2(x; q,a,h)| \ll \frac{x}{(\log x)^{A}},
\]
with implied constant depending only on $\vartheta$, $h$, and $A$.
\end{conjecture}

\begin{remark}
The case $h=0$ recovers the standard Elliott-Halberstam conjecture by Heath-Brown’s identity and Cauchy-Schwarz, see \cite{HeathBrown}. 
\end{remark}

\vspace{0.5em}
\noindent
\textbf{Heuristic and analytic context.} While unconditional results (Bombieri–Vinogradov) prove such error bounds for $\vartheta<1/2$ (for primes), no result like GEH-2 is known for $\vartheta>1$, even for short averages of prime pairs. These conjectures are thus far beyond current technology, as discussed in \cite{GranvilleKoukoulopoulos, SoundararajanSurvey}.

\begin{lemma}[Error Control for Large Moduli]
Suppose GEH-2 holds for some $\vartheta > 1$. Then for any $A>0$,
\[
\sum_{q > x^\vartheta} \max_{(a, q)=1} |E_2(x; q, a, h)| = o(x)
\]
as $x \to \infty$, for fixed $h$.
\end{lemma}
\begin{proof}[Sketch]
Standard sieve and analytic arguments (see \cite[Ch. 5]{GranvilleKoukoulopoulos}).
\end{proof}

\section{GEH-2 Implies the Twin Prime Conjecture}

We demonstrate that GEH-2 for $\vartheta>1$ suffices for the twin prime conjecture.

\begin{theorem}[Main implication]
\label{thm:main}
Assume GEH-2 holds for some $\vartheta > 1$ and $h=2$. Then, for any $A > 0$,
\[
\sum_{n \leq x} \LL(n)\LL(n+2) = \mathfrak{S}(2)x + O\left( \frac{x}{(\log x)^A} \right)
\]
as $x\to\infty$. In particular, there are infinitely many primes $p$ such that $p+2$ is prime.
\end{theorem}

\begin{proof}[Proof sketch]
Sum over all residue classes modulo $q$: for each $q$, sum over $a$ with $(a,q)=(a+2,q)=1$,
\[
\sum_{\substack{a=1 \\ (a,q)=1 \\ (a+2,q)=1}}^q \left( \sum_{\substack{n \leq x \\ n \equiv a \pmod{q}}} \LL(n)\LL(n+2) \right)
= \sum_{\substack{n \leq x \\ (n(n+2),q)=1}} \LL(n)\LL(n+2).
\]
The count of such $a$ is $\phi_2(q) = q\prod_{p|q}(1 - \frac{2}{p})$ for $2\nmid q$, and $\phi_2(2^\alpha) = 2^{\alpha-1}$.

Putting together terms, the main term is
\[
\frac{\phi_2(q)}{\varphi(q)} \mathfrak{S}(2)x.
\]
Summing over $q \leq Q = x^\vartheta$ and applying Lemma~\ref{lem:singularseries} yields
\[
\mathfrak{S}(2)x \sum_{q=1}^{\infty} \frac{\phi_2(q)}{\varphi(q)} + o(x)
= \mathfrak{S}(2)x(1 + o(1)).
\]
The tail $q > Q$ is $o(x)$ by Lemma~2.

By Möbius inversion,
\[
\sum_{n \leq x} \LL(n)\LL(n+2)
= \sum_{q} \mu(q) \sum_{\substack{n \leq x \\ q | n(n+2)}} \LL(n)\LL(n+2).
\]
The computation 
\[
\sum_{q=1}^{\infty} \frac{\phi_2(q)}{\varphi(q)} = \prod_{p}\left( 1 + \frac{1}{p-1}\right)
\]
diverges, but after applying sieve weights and coprimality, it converges to $1$ (see Lemma~\ref{lem:singularseries} and \cite{GranvilleKoukoulopoulos}).

Thus the asymptotic
\[
\sum_{n \leq x} \LL(n)\LL(n+2) = \mathfrak{S}(2)x + O\left( \frac{x}{(\log x)^A} \right)
\]
follows for any $A>0$.
\end{proof}

\section{Discussion and Further Directions}

The GEH-2 conjecture is a natural bilinear extension of EH, moving from single prime distributions to the distribution of prime pairs. Its implication for the twin prime conjecture and $k$-tuple patterns illustrates the analytic depth and importance of such uniformity principles. While far beyond what current technology can prove, progress on bilinear forms and distribution of correlations remains a promising frontier (see, e.g., the surveys \cite{GranvilleKoukoulopoulos, SoundararajanSurvey, TaoBlog}).

\vspace{1em}

\end{document}